\documentclass[letterpaper]{article}
\pdfoutput=1

\title{Sets resilient to erosion}
\date{June 28, 2008}
\author{Wesley Pegden\footnote{
Department of Mathematics, 
Rutgers University (New Brunswick), 
110 Frelinghuysen Rd.,
Piscataway, NJ 08854-8019. 
Email: pegden@math.rutgers.edu}
}

\usepackage{epic,eepic,xcolor}
\usepackage{amsmath,amsthm,amssymb,amscd,latexsym}
\usepackage{slashbox}
\usepackage{cite,pstricks,pst-plot,multido}
\usepackage{pst-xkey}
\usepackage{pstricks-add}
\usepackage{url}
\usepackage{subfigure}
\usepackage{epsfig}

\newcommand{\comments}[1]{}

\renewcommand{\ldots}{\dots}

\newcommand{\hhh}{{\cal H}}

\newcommand{\bc}{\mathrm{bc}}

\newcommand{\sbs}{\subset}
\newcommand{\abs}[1]{\lvert #1 \rvert}

\newcommand{\st}{\,\vrule\,}

\newcommand{\Z}{\mathbb{Z}}

\newcommand{\R}{\mathbb{R}}

\newtheorem{theorem}{Theorem}[section]

\newtheorem{lemma}[theorem]{Lemma}

\newtheorem{observ}[theorem]{Observation}

{
\theoremstyle{definition}
\newtheorem{definition}[theorem]{Definition}

\newtheorem{question}[theorem]{Question}
}

{
\theoremstyle{remark}

}

   \newtheoremstyle{example}{\topsep}{\topsep}%
     {}
     {}
     {\bfseries}
     {}
     {\newline}
     {\thmname{#1}\thmnumber{ #2}\thmnote{ #3}}

   \theoremstyle{example}
   \newtheorem{example}[theorem]{Example}

\newcommand{\ep}{\varepsilon}

\newcommand{\stm}{\setminus}
\newcommand{\bound}{\partial}

\begin{document}
\maketitle
\begin{abstract}
  The \emph{erosion} of a set in Euclidean space by a radius $r>0$ is the subset of $X$ consisting of points at distance $\geq r$ from the complement of $X$.  A set is \emph{resilient} to erosion if it is similar to its erosion by some positive radius.  We give a somewhat surprising characterization of resilient sets, consisting in one part of simple geometric constraints on convex resilient sets, and, in another, a correspondence between nonconvex resilient sets and scale-invariant (\emph{e.g.}, `exact fractal') sets.  

\end{abstract}
\section{Introduction}

\begin{figure}[b]
\vspace{-.5cm}
  \begin{center}
    \includegraphics{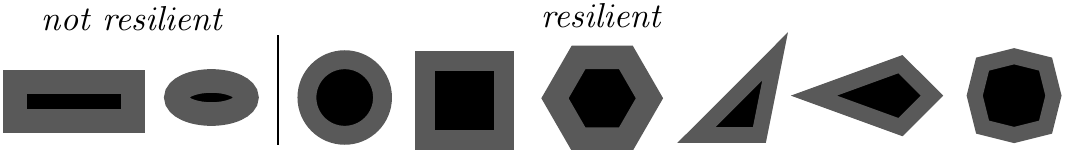}
  \end{center}
  \caption{Some erosions of bounded shapes in $\R^2$.  The area in gray is what is removed by an erosion operation.\label{polys}}
\end{figure}

Given a subset $X$ of $\R^n$, define the \emph{erosion} $e_r(X)$ of $X$ by the radius $r$ as the set of points of $X$ at distance $\geq r$ from the complement $X^C$ of $X$.  So we have
\begin{eqnarray}
&&  e_r(X)=X\stm \bigcup_{y\in X^C} B(r,y),\label{crX}
\end{eqnarray}
where $B(r,y)$ denotes an open ball of radius $r$ about $y$.

It turns out that this operation has been studied from a practical standpoint, as a model of pebble erosion.  For example, V\'arkonyi et al. \cite{vds} have studied the `typical' limit shapes under this and related operations to explain the distribution of pebble shapes found in in different kinds of natural environments (see also \cite{erop}).  We are interested in this operation because of a question of a more theoretical nature, namely:  \textbf{Given a set $X\sbs \R^n$ and a radius $r>0$, when is it true that $e_r(X)$ is equivalent to $X$ under a Euclidean similarity transformation?} 

When this is the case, we say that $X$ is \emph{resilient to erosion by the radius $r$}.  When this is true for at least one positive $r$, we say $X$ is \emph{resilient}.  We answer this question by giving a complete characterization of resilient sets.  In spite of the fact that the question is already quite interesting and natural in the 2-dimensional case, we will see that the characterization we give applies in any number of dimensions.

The characterization we give is a bit unusual in its form; in a certain sense, it is one part convex geometry, one part `fractal' geometry.   While we characterize convex resilient sets (we will see this includes all bounded resilient sets) with simple geometric constraints, the rest of the characterization is a correspondence between certain resilient sets and `scale-invariant' sets.  For example, Figure \ref{f.ckoch} shows a resilient set related to the Koch snowflake.  Notice that this set `gets bigger' when it is eroded.

\begin{figure}
  \psset{unit=3.3cm}

\begin{center}
\subfigure[\label{f.koch} A portion of an unbounded (and scale-invariant) version of Koch's curve.]{
\includegraphics{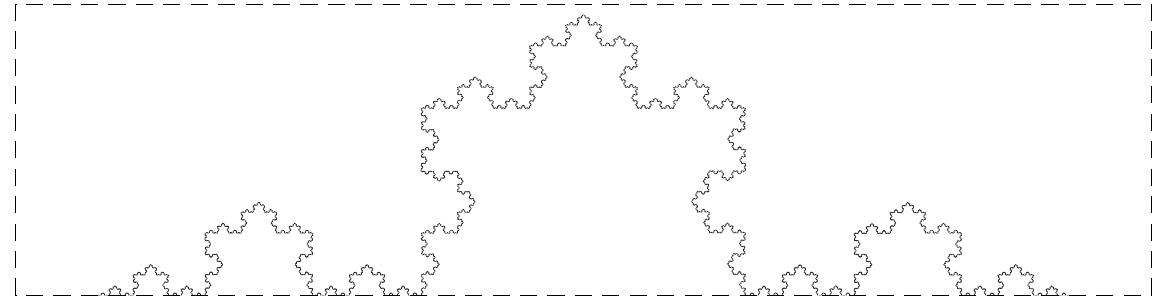}
}

\subfigure[\label{f.ckoch} Part of an unbounded resilient set derived from the unbounded version of Koch's snowflake.  The area removed by two successive erosion operations is shown in two different shades of gray.  Note that erosion makes it `bigger'.]{
\includegraphics{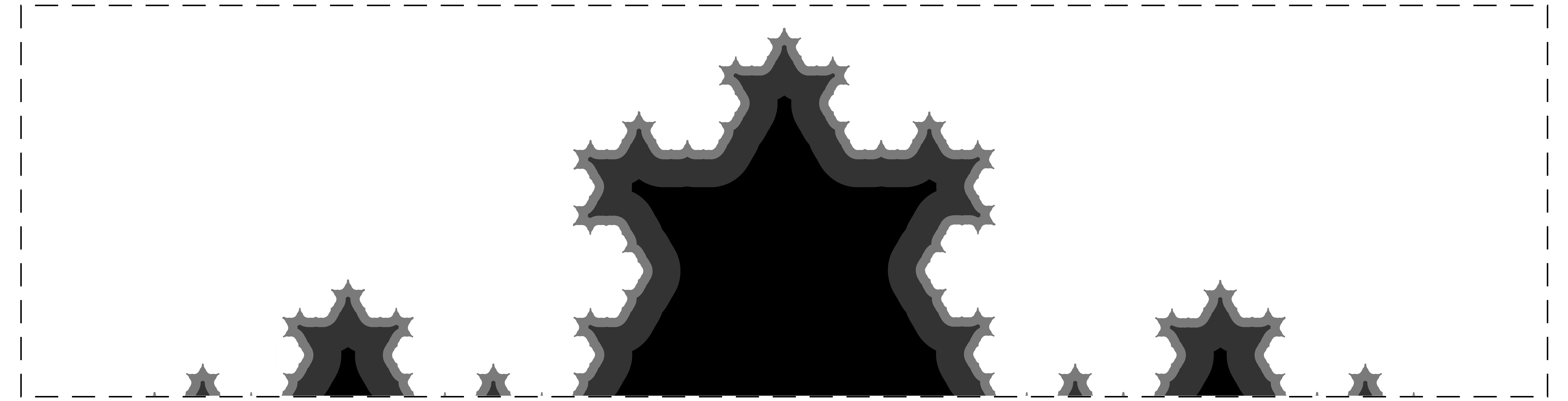}
}
\end{center}
  \caption{Getting a resilient set from the Koch snowflake.}  
\end{figure}


Before proceeding, lets consider some simple bounded examples from the plane.  It is easy to see that a closed ball of positive diameter is resilient, as is the body of a square---more generally, the body of any regular polygon.  (We mean `body' to include the boundary, so these are closed sets.)  One can also check that the body of any triangle is resilient, since it can be eroded to a smaller triangle with the same set of angles.  
On the other hand, any rectangle which is not a square is not resilient to erosion by any radius, since the ratio of the two side lengths will change under erosion.

We will see shortly that it is possible to precisely characterize all bounded resilient sets in simple geometric terms.  Notice first from line (\ref{crX}) that an erosion of \emph{any} set is always closed, thus all resilient sets are closed.  We will see that the bounded resilient sets of $\R^n$ are all convex, and are exactly those closed, convex bounded sets which have an inscribed ball, in a sense we will make precise in a moment.  For now notice that all the resilient polygons in Figure \ref{polys} have inscribed circles.

Call a closed, convex set a \emph{convex body}.  A closed half-space $H_x$ containing the convex body $X$ is a \emph{supporting half-space} at a point $x$ on the boundary of $X$ if $x$ also lies on the boundary of $H_x$.  A convex body in $\R^n$ is the intersection of all of its supporting half-spaces.  More is true: Call a point $x$ on the boundary of $X$ \emph{regular} if it has a unique supporting half-space $H_x$, and call a supporting half-space regular if it is supporting at at least one regular point of $X$.  Then in fact, a convex body $X\sbs \R^n$ is the intersection of its regular supporting half-spaces.  (See for example \cite{comgeom} as a reference.)   A \emph{supporting hyperplane} of $X$ is just the boundary of one of its supporting half-spaces, and similarly, a \emph{regular supporting hyperplane} is the boundary of one of its regular supporting half-spaces.  Notice that the only regular supporting lines of a polygon are those that coincide with a side of the polygon.  This suggests our definition of an inscribed ball of a convex body:

\begin{definition}
  A (closed) ball $B$ is \emph{inscribed} in the convex body $X$ if $B\sbs X$ and $B$ intersects all the regular supporting hyperplanes of $X$.
\label{inscribed}
\end{definition}
\noindent Notice that for polygons, Definition \ref{inscribed} coincides for the convex body of a polygon with the definition of an inscribed circle for the boundary of the polygon.  In 2 dimensions, Definition \ref{inscribed} requires that all of the `straight' parts of the boundary of the convex body be tangent to the inscribed ball, while all other parts of the boundary lie on the ball itself.  (See Figure \ref{example}.)  Our characterization of bounded resilient sets is given by the following theorem.

\begin{figure}[t]
  \begin{center}
    \includegraphics{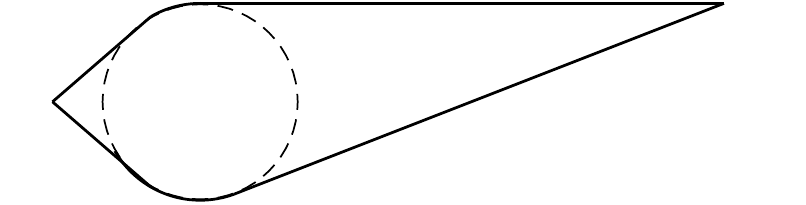}
  \end{center}
\caption{Another example of the boundary of a resilient set from $\R^2$.  The dashed circle is inscribed in the sense of Definition \ref{inscribed}\label{example}.}
\end{figure}

\begin{theorem}
  A bounded set $X\in \R^n$ is resilient to erosion by some radius $r>0$ if and only if it is a convex body with an inscribed ball of radius $>r$. 
\label{coll}
\end{theorem}

\noindent \textbf{Proof idea (proof for polytopes):} Let $X$ be a bounded polytope in $\R^n$.  For the `if' direction of the theorem, notice that $e_r(X)$ is a polytope whose regular supporting hyperplanes are parallel to to the regular supporting hyperplanes of $X$ and, like regular supporting hyperplanes of $X$, are at constant distance from the center of the inscribed ball of $X$.  This implies that $e_r(X)$ is similar to $X$.

For the other direction, we have that $e_r(X)=\sigma(X)$ for a similarity transformation $\sigma$.  First note that $X$ must be convex; otherwise, $e_r(X)$ is not even a polytope.  So long as $X$ is convex, $e_r(X)$ is another polytope, and each regular supporting hyperplane $H$ of $e_r(X)$ is parallel to, and at distance $r$ from, a regular supporting hyperplane $\bar H$ of $X$.  Since $X$ and $e_r(X)$ are similar, they have the same finite number of sides, and so every regular supporting hyperplane of $X$ is such a $\bar H$.  Letting $p$ denote the fixed point of $\sigma$, we have $d(H,p)=d(\bar H,p)-r$ for all regular supporting hyperplanes $H$ of $e_r(X)$.   Therefore, for the lists of distances $\{d(H,p)\}$ and $\{d(\bar H,p)\}$ to be the same up to a scaling factor, we must have that they are both lists of constant distances; i.e., that the regular supporting hyperplanes of each lie at constant distances from $p$.

\vspace{1em}

The proof of Theorem \ref{coll} is the subject of Section \ref{s.bounded}.  The first hole in the above `proof idea' when applied to general bounded sets is showing that they must be convex.  The second (more annoying) snag is that when the number of regular supporting hyperplanes is not finite, it does not seem straightforward to argue that all regular supporting hyperplanes of $X$ must be parallel to regular supporting hyperplanes of $e_r(X)$ (in fact, we carry out the complete proof of Theorem \ref{coll} without showing this to be the case, though it is an immediate consequence of the Theorem).  This difficulty arises because of the prospect of similarity under nonhomothetic transformations (\emph{i.e.}, transformations that include nontrivial rotations).  Though it may seem intuitively natural that sets which are resilient must be so under homothetic transformations, we will see in Section \ref{s.nonconvex} that, without some assumptions like boundedness or convexity, there are in fact sets $X\sbs \R^n$ where $e_r(X)=\sigma(X)$, but only for transformations $\sigma$ which include an irrational rotation.


As we will see in Section \ref{s.nonconvex}, unbounded resilient sets may be nonconvex and quite complicated geometrically, and unlikely to satisfy any simple geometric characterization like that given in Theorem \ref{coll}. Nevertheless, we \emph{can} prove a simple geometric characterization of (possibly unbounded) resilient sets which \emph{are} convex; this is the subject of Theorems \ref{t.dcoll}, \ref{t.inccoll}, and \ref{t.isomcoll}.

Resilient sets which are nonconvex behave in a counterintuitive way: eroding them makes them `bigger', in that the corresponding similarity transformation is distance-increasing.  To complete our characterization, we will see in Section \ref{s.sic} that a set is resilient to erosion in this distance-increasing way if and only if it is the erosion of a scale-invariant set (one which is self-similar under non-isometric transformations), giving a surprising connection with `exact fractals'.

\vspace{1ex}

There is another operation closely related to the erosion operation, which we call the \emph{expansion} by a radius $r$, defined as
\begin{equation}
  E_r(X)=\bigcup_{x\in X}B(r,x).
  \label{ErX}
\end{equation}
Again $B(r,x)$ denotes an open ball, thus $E_r(X)$ is always an open set.  Notice that we have $e_r(X)=E_r(X^C)^C$, thus, in general, the erosion operation is equivalent to the expansion operation by taking complements.  Nevertheless, the family of \emph{convex} expansion-resilient sets is much less rich than that of convex erosion-resilient sets.  For example, it is not too hard to see that the only bounded expansion-resilient sets are open balls.  We give the characterization of all convex expansion-resilient sets in Section \ref{s.exp}.

\section{The bounded case}
\label{s.bounded}
We first show the `if' direction of Theorem \ref{coll}: as noted above, a convex body $X$ containing a ball $B$ of radius $R$ intersecting its regular supporting hyperplanes is determined up to similarity by the selection of the positions of the points on the ball where these intersections occur.  The important point is that, up to the scaling of the ball, these positions are unchanged by the erosion operation.

To make this precise, observe that the erosion $e_r(X)$ of  $X$ by a radius $r<R$ contains the ball $B'=e_r(B)$ of radius $R'=R-r$.   Also, note that $e_r(X)$ is the intersection of the erosions $e_r(H)$ of the regular supporting half-spaces $H$ of $X$. The erosion $e_r(H)$ by $r$ of any supporting half-space $H$ of $X$ tangent to $B$ will be tangent to $B'$, and the position of the point where the boundary of $(e_r(H))$ intersects  $B'$ is the same as the point where the boundary of $H$ intersects $B$, apart from the scaling of the ball, since the point where a half-space is tangent to a ball is determined by the orientation of the half-space.  Thus $e_r(X)$ is the intersection of some half-spaces which are the images of the regular supporting half-spaces of $X$ under some fixed dilation, and so $e_r(X)$ is similar to $X$ under that same dilation.

For the other direction, we will first work to show that if $X$ is resilient and bounded, then it must be convex.  (This is not the case if we drop the unbounded requirement.)  To do this, it suffices to show that if a nonconvex set undergoes erosion, some distance associated with it increases (making it impossible that it is similar to the original set under a distance-decreasing similarity transformation).  The associated `distance' we use is the \emph{ball-convexity}:

\begin{definition}
  The \emph{ball-convexity} $\bc(X)$ of a closed set $X\sbs \R^n$ is the supremum of radii $R$ for which the complement of $X$ is a union of open balls of radius $R$.
\end{definition}
\noindent The ball-convexity of a set provides a kind of measure of its convexity: note that any convex body $X$ satisfies $\bc(X)=\infty$, since the complement of any convex body can be written as the union of open half-spaces, which can be in turn written as the unions of balls of arbitrarily large radii.

What about the converse?  It is not true in general that a set with infinite ball-convexity is convex: for example, any closed set $X\sbs \R^2$ lying entirely on some line has infinite ball convexity, even though many such sets are not convex.  This the only kind of counterexample, however:
\begin{observ}
  A closed set $X\sbs \R^n$ which does not lie in any $(n-1)$-dimensional subspace is convex if and only if $\bc(X)=\infty$. 
\label{o.bcc}
\end{observ}
\begin{proof}
 As already pointed out, the `only if' direction is clear.  For the other direction,  assume $X$ is not convex.  So there are points $x_1,x_2\in X$ such that the line segment $(x_1,x_2)$ includes a point $q$ not in $X$.  Since $X$ is closed, in fact there is a ball $B(\ep,q)$ about the point $q$ which contains no points of $X$.

Since $X$ doesn't lie entirely in some $(n-1)$-dimensional subspace, we can choose additional points $x_3,x_4,\dots x_{n+1}\in X$ so that the convex hull of all the points $x_1,x_2,x_3,\dots,x_{n+1}$ is an $(n+1)$-simplex.  Now choose a point $p$ from $B(\ep,q)$ which lies inside the $(n+1)$-simplex (but not in $X$).

Let $R$ the maximum of the radii of the $n+1$ $(n-1)$-spheres determined by $(n+1)$-tuples from the set $\{p,x_1,x_2,\dots,x_{n+1}\}$.  Then we have that $\bc(X)\leq R$, since a ball of radius $>R$ cannot include the point $p$ without including one of the points $x_i\in X$.  (This situation for $n=2$ dimensions is shown in Figure \ref{f.bc}.)   In particular, $\bc(X)<\infty$, as desired.
\end{proof}

\begin{figure}
  \begin{center}
    \includegraphics{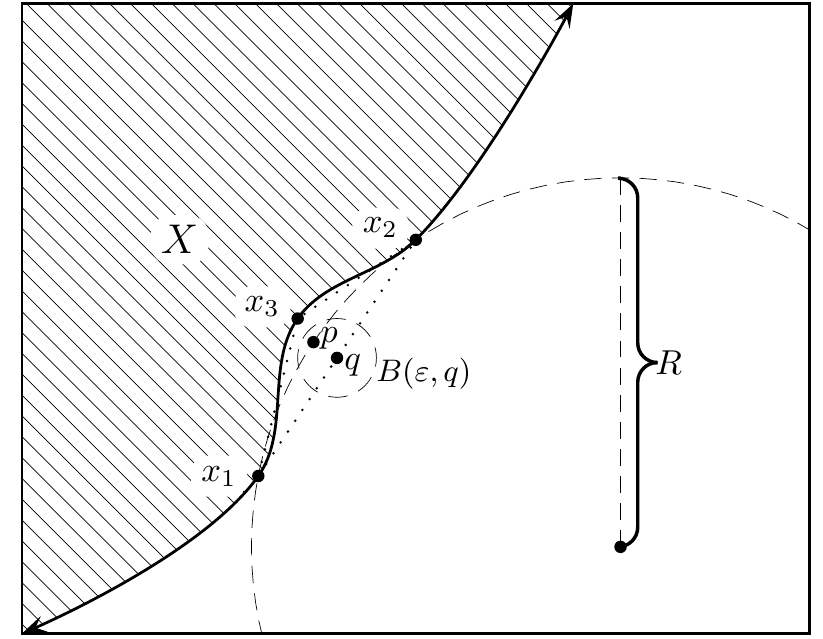}
  \end{center}
  \caption{Proving the `if' direction of Observation \ref{o.bcc} for the dimension $n=2$; the distance $R$ is a upper bound on the ball-convexity of $X$; balls of any larger radius cannot include the point $p$ without also including points of $X$. \label{f.bc}}
  
\end{figure}
\noindent Armed with the concept of ball convexity, it easy to show:
\begin{lemma}
  Let $e_r(X)=\sigma(X)$ for some $r>0$ and some similarity transformation $\sigma$.  Then $X$ must be convex, unless the transformation $\sigma$ is distance-increasing.
\label{l.convex}
\end{lemma}
\begin{proof}
  Since $X$ is resilient, it is closed and does not lie in any $(n-1)$-dimensional subspace.  Also, if $\bc(X)$ is finite, then we have that $\bc(e_r(X))\geq \bc(X)+r$.  Thus we are done by Observation \ref{o.bcc}, since for $\bc(X)$ to be finite, $\sigma$ would have to be distance-increasing.
\end{proof}
The similarity transformation corresponding to the erosion of a bounded resilient set is always distance-decreasing since, for example, the diameter decreases upon erosion.  Thus Lemma \ref{l.convex} implies that bounded resilient sets must be convex.  Since a convex set is the intersection of its supporting half-spaces (and even of just its regular ones), we can therefore study the erosion operation by examining its effect supporting half-spaces/hyperplanes.  

First note that the erosion of an intersection equals the intersection of the erosions: 
$
e_r(\bigcap A_\alpha)= \bigcap e_r(A_\alpha)
$
 for any family of subsets $A_\alpha\sbs \R^n$ (this is easily deduced from line (\ref{crX})).  Since a convex set $X$ is the intersection of its supporting half-spaces, we have for any regular supporting half-space $H$ of $X$ that $e_r(X)\sbs e_r(H)$.  Thus, if some point on the boundary of $e_r(X)$ is at distance $r$ from the boundary of $H$, we have that $e_r(H)$ is in fact a supporting half-space to $e_r(X)$ at the point $x$.  Of course, any point $x$ on the boundary of $e_r(X)$ must lie at distance $r$ from \emph{some} point $\bar x$ on the boundary of $X$; we see that if $x$ has a unique supporting half-space, then so must $\bar x$:



\begin{observ}
  If $X\sbs \R^n$ is closed and convex, then any regular point $x$ on the boundary of $e_r(X)$ is at distance $r$ from a (unique) $\bar x$ on the boundary of $X$ which is regular.  The unique supporting hyperplane at $x$ of $e_r(X)$ is a parallel to the unique supporting hyperplane of $X$ at $\bar x$.\qed
\label{regtreg}
\end{observ}

Theorem \ref{coll} will imply that any bounded resilient set $X$ is in fact resilient to erosion by all radii $r$ for which $e_r(X)$ has positive diameter. For the proof of the Theorem, we want to at least know that a bounded resilient set can be eroded to an arbitrarily small copy of itself.  First observe the following:
\begin{observ}
  For any similarity transformation $\sigma$ with scaling factor $\alpha$, we have  $\sigma(e_r(X))=e_{\alpha r}(\sigma(X))$.\qed
\label{o.commutes}
\end{observ}
\noindent  Observation \ref{o.commutes} captures the way in which the erosion of a resilient set is again resilient to erosion.  If $X$ is resilient to erosion by the radius $r$, then $e_r(X)=\sigma(X)$, so that Observation \ref{o.commutes} gives us that $e_{\alpha r}(e_{r}(X))=\sigma^2(X)$.  By the triangle inequality, $e_{\alpha r}(e_{r}(X))$ is just $e_{r+\alpha r}(X)$, so we get that in addition to being resilient to erosion by the radius $r$, $X$ is also resilient to erosion by the radius $r+\alpha r$, with corresponding similarity transformation $\sigma^2$.  Repeating this argument in the natural way gives us the following observation:
\begin{observ}
If $e_r(X)=\sigma(X)$ for some similarity transformation $\sigma$ with scaling factor $\alpha$,  then it is resilient to erosion by 
\begin{equation}
r_i=\sum\limits_{0\leq k<i} r\alpha^k  
\label{e.ps}
\end{equation}
for all $i$; the corresponding similarity transformation is $\sigma^i$.\qed
\label{sup}
\end{observ}

Our final piece of preparation concerns spherical subsets which are isomorphic to proper subsets of themselves:
\begin{observ}
  Let $A\sbs B$ be subsets of the sphere $S^{n-1}$, such that $\phi(A)=B$, where $\phi$ is some isometry of the sphere.  Then $A$ is dense in $B$.
\label{o.dense}
\end{observ}
\begin{proof}
 Otherwise, let some neighborhood $U$ of a point $b\in B$ contain no points of $A$.  Then $\phi^{-1}(U)$ contains no points of $B$, and so again no points of $A$.  But then $\phi^{-2}(U)$ contains no points of $B$ and so no points of $A$, \emph{etc.}: we have that $\phi^{-k}(U)$ contains no points of $B$ (or $A$) for any $k\in \Z^+$.  

On the other hand, for the contradiction, we claim that for any $\ep$, we can always find some $k\in \Z^+$ so that $\phi^{-k}(b)$ is at distance $<\ep$ from $b$.  Certainly, among all of the images $\phi^{-1}(b),\phi^{-2}(b),\phi^{-3}(b)\dots$, there must be a pair $\phi^{-i}(n),\phi^{-j}(n)$ ($i<j$) at distance $<\ep$ (possibly, the distance is 0).  Applying $\phi^i$, we have that $b$ is at distance $<\ep$ from $\phi^{-j+i}(b)$.
\end{proof}

We are now ready to prove Theorem \ref{coll}.  In fact, we are ready to prove a more general statement, which includes Theorem \ref{coll} as a special case, but also makes up an important part of the characterization of (possibly unbounded) convex resilient sets:

\begin{theorem}
\label{t.dcoll}
  Let $X$ be a (possibly unbounded) subset of $\R^n$.  The we have $\sigma(X)=e_r(X)$ for some $r$ and some distance-decreasing similarity transformation $\sigma$ if and only if $X$ is a convex body with an inscribed ball of radius $>r$.
\end{theorem}
\begin{proof}
 Given that $e_r(X)=\sigma(X)$, Observation \ref{sup} gives us that $X$ is resilient to erosion by the radii $r_1,r_2,r_3,\dots$ defined by the partial sums in line (\ref{e.ps}), with corresponding similarity transformations $\sigma,\sigma^2,\sigma^3,\dots,$ respectively.  Let $R=\sup (r_i)=\sum_{k=0}^\infty r\alpha^k$, and let $p$ be the fixed point of the similarity transformation $\sigma$.  $X$ is convex by Lemma \ref{l.convex}.  We will prove Theorem \ref{coll} by demonstrating that the boundaries of all regular supporting half-spaces $H$ of $X$ lie at distance $R$ from the point $p$.  Note that any supporting hyperplane of $X$ must lie at distance $\geq R$ from $p$, since $e_{r_i}(X)$ must contain the point $p$ for every $r_{i}$ (since $p$ is fixed by $\sigma$); we need to show that all of the regular supporting hyperplanes of $X$ lie at distance $\leq R$ from $p$.

\begin{figure}[t]
    \psset{unit=.65cm}
  \begin{center}
  \subfigure[\label{p.easy} The situation in the case where $\sigma$ is just a homothety (\textit{i.e.,} it does not include a rotational component).  Here $H$ is parallel to $H'$, so $H=\bar H$.]{
    \includegraphics{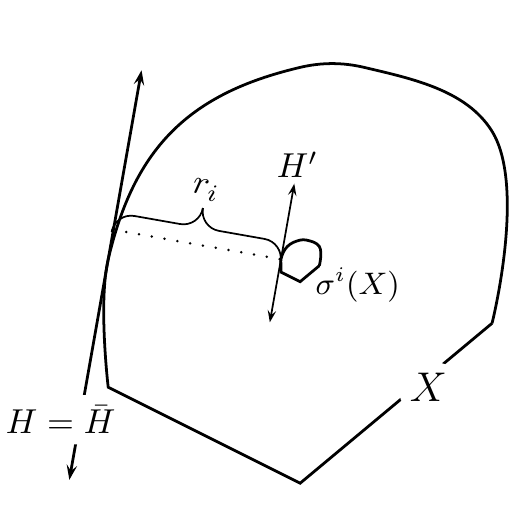}
}
\hspace{.3cm}
  \subfigure[\label{p.hard} The situation where $\sigma$ may include a rotation.  Notice that, as drawn, $H$ is \emph{not} parallel to any regular supporting hyperplane of $e_{r_i}(X)$.]{
    \includegraphics{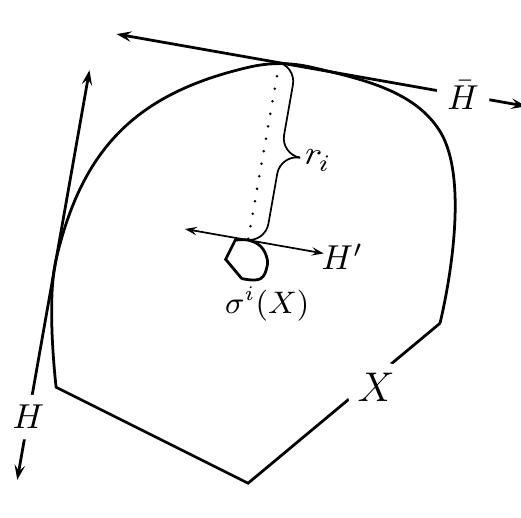}
}
  \end{center}
  \caption{Proving Theorem \ref{t.dcoll}. In both cases, $H'$ is the image of $H$ under $\sigma^i$.  
}
\end{figure}

Since we are not assuming that $X$ is bounded, we do not assume even that the regular supporting hyperplanes of $X$ are at bounded distance from $p$.  To keep this from causing problems, we fix some number $D>R+1$, and focus on the set $\hhh_D$ of regular supporting hyperplanes of $X$ lying within distance $D$ of $X$.  We will show that for any $H$ in $\hhh_D$, we actually have that $H$ lies at distance $\leq R$ from $p$.  Since $D>R+1$ can be chosen arbitrarily, this implies the Theorem.

Because hyperplanes in $\hhh_D$ lie at bounded distance from $p$, we can choose $i$ sufficiently large so that all the images $\sigma^i(H)$, ($H\in \hhh_D$) lie within any positive distance $\ep$ from the point $p$.  Note that each such hyperplane $H'=\sigma^i(H)$ ($H\in \hhh_D$) is a regular supporting hyperplane of $e_{r_i}(X)$.  Thus, by Observation \ref{regtreg}, there is a regular supporting hyperplane $\bar H$ of $X$ lying parallel to and at distance $r_i$ from $H'$.  Since $H'$ is at distance $<\ep$ from $p$, $\bar H$ is at distance $<r_i+\ep$ from $p$, and thus at distance $<R+\ep$.  Let $\hhh_D^i=\{\bar H\st H\in \hhh_D\}$.  Note that by choosing $\ep<1$, we have $R+\ep<D$, thus, in particular, $\hhh_D^i\sbs \hhh_D$.


%
Since $\hhh_D^i$ consists of hyperplanes within distance $R+\ep$ of $p$, we would be done (by letting $\ep$ go to 0) if we could assert that $\hhh_D=\hhh_D^i$ for all $i$.  This is certainly the case if $\sigma$ doesn't include a rotation and consists just of the homothety $x\mapsto \alpha(x-p)+p$.  In this case, given any regular supporting hyperplane $H$ of $X$, its image $H'=\sigma(H)$ is a regular supporting hyperplane of $e_{r_i}(X)$ \emph{which is parallel} to $H$, and therefore we have that $H=\bar H$ (this situation is shown in Figure \ref{p.easy}. 

Since we do not assume that $\sigma$ is homothetic, we allow that $\hhh_D^i\subsetneq \hhh_D$; \emph{i.e.}, that there are regular supporting hyperplanes of $X$ not parallel to any regular supporting hyperplane of $e_{r_i}(X)$.  Figure \ref{p.hard} shows this hypothetical situation.  Let $\nu(H)$ denote the unit normal vector of any supporting hyperplane $H$ of $X$, (oriented towards $X$, say).  Since $\sigma$ is distance-decreasing, it has a fixed point $p$, and can be written as the product of a homothety and an isometry, which both fix $p$.   The isometry induces an isometry $\phi$ of the unit sphere $S^{n-1}$, and we have that $\nu(\bar H)=\phi(\nu(H))$.  Thus, the set of the orientations of regular supporting hyperplanes of $X$ which are parallel to regular supporting hyperplanes of $e_{r_i}(X)$ is the image under the rotation $\phi$ of the set of the orientations of \emph{all} regular supporting hyperplanes of $X$.  Thus by Observation \ref{o.dense}, the set of orientations of regular supporting hyperplanes $\bar H$ which are parallel to regular supporting hyperplanes of $e_{r_i}(X)$ is dense in the set of orientations of all regular supporting hyperplanes of $X$.  In other words, for any regular supporting hyperplane $H$ of $X$ and for any $\ep'>0$, we can find another regular supporting hyperplane $\bar H_0$ of $X$ at dihedral angle $\theta<\ep'$ to $H$, which is parallel to some regular supporting hyperplane $H_0$ of $e_{r_i}(X)$, and thus lies at distance $<R+\ep$ from $p$.

\begin{figure}
  \begin{center}
    \includegraphics{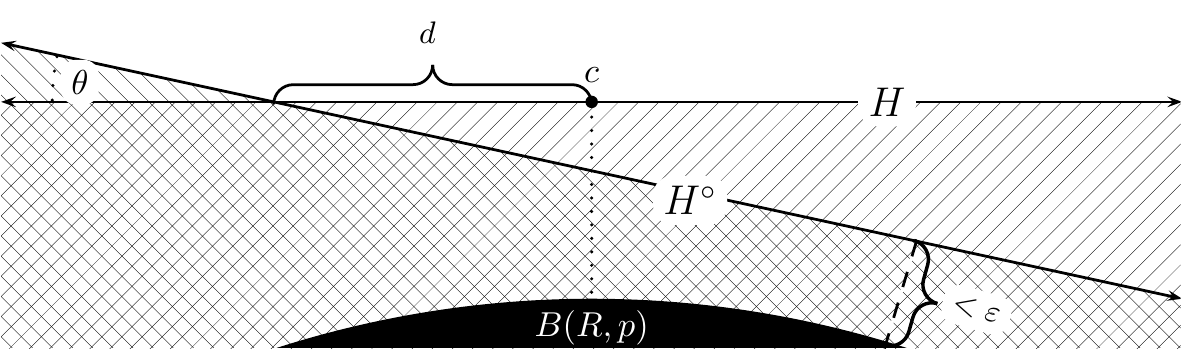}
  \end{center}
  \caption{Finishing the proof of Theorem \ref{t.dcoll}\label{f.proofsep}.  The shaded regions are the half-spaces corresponding to $H$ and $\bar H_0$.  $X$ lies somewhere in their intersection.}
  
\end{figure}

To finish the proof, assume that some hyperplane $H\in \hhh_D$ of $X$ is at distance $>\ep$ from the ball $B(R,p)$ of radius $R$ centered at $p$.  We have that there is a hyperplane $H^\circ\in \hhh_D^i$ of $X$ at an arbitrarily small dihedral angle $\theta$ to $H$.  Since $H^\circ\in\hhh_D^i$, we have that $H^\circ$ is within distance $R+\ep$ of $p$.  Referring to Figure \ref{f.proofsep}, we let $c$ be the point on the boundary of $H$ which is closest to the point $p$, and let $d$ be the distance between $c$ and the intersection $H\cap H^\circ$.  Any points on $H$ at distance $<D$ from $c$ lie outside of supporting half-space of $X$ whose boundary is $H^\circ$, and so outside of $X$ as well.  By choosing $H^\circ$ so that the angle $\theta$ is arbitrarily large, we can make $d$ arbitrarily large; thus no points on the boundary of $H$ lie on $X$, and $H$ cannot be a supporting half-space to $X$ after all.  We conclude that all regular supporting half-spaces of $X$ lie at distance $R$ from $p$.
\end{proof}

\section{The general convex case}
\label{s.unbounded}
In this section we consider the case of general (\emph{i.e.}, possibly unbounded) convex resilient sets; our aim is to characterize these sets in simple geometric terms.  Theorem \ref{t.dcoll} already provides a characterization for the case where $e_r(X)=\sigma(X)$ and $\sigma$ is distance-decreasing.  Two cases remain for our geometric characterization of convex resilient sets: the case where $\sigma$ is an isometry, and the case where $\sigma$ is distance-increasing.  These cases are covered in Section \ref{s.cacc}.  First, we consider some examples from $\R^3$ which demonstrate that Theorem \ref{t.dcoll} isn't the complete story, and that all of these possibilities really can occur.  (It is not too hard to check that at least 3 dimensions are required to have examples of convex resilient sets not covered by Theorem \ref{t.dcoll}.)

\subsection{Unbounded convex examples in $\R^3$}
\label{R3counter}
In this section we give an example of a convex set $X_1$ such that $X_1$ is resilient under a distance-decreasing similarity transformation and so has an inscribed ball, $e_r(X_1)$ is not resilient, and such that $e_{2r}(X_1)$ is resilient under a distance-increasing similarity transformation (and has no inscribed ball).

Let $X_1$ be the `bottomless tent' depicted in Figure \ref{tent}.  $X_1$ is the unbounded intersection of four half-spaces $H_i$, chosen so that $\bound H_1\cap \bound H_3\cap X_1$ is a line segment (here $\bound H_i$ denotes the boundary), while $H_i\cap H_j\cap X_1$ is a ray whenever $j\equiv i+1\pmod 4$.  The half-spaces are chosen so that the dihedral angle between $H_2$ and $H_4$ is greater than that between $H_1$ and $H_3$, and such that any horizontal cross section is rectangular.

\begin{figure}[t]
  \begin{center}
   \subfigure[
$X_1$\label{tent}
]{
     \includegraphics{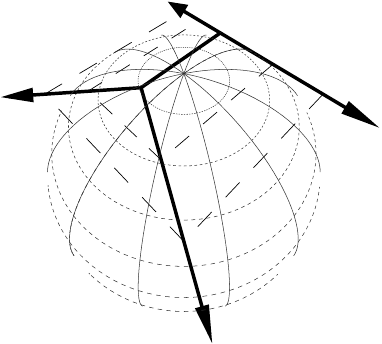}
   }
\hspace{-.5cm}
    \subfigure[
$X_2=e_r(X)$\label{ibltetra}]{
      \includegraphics{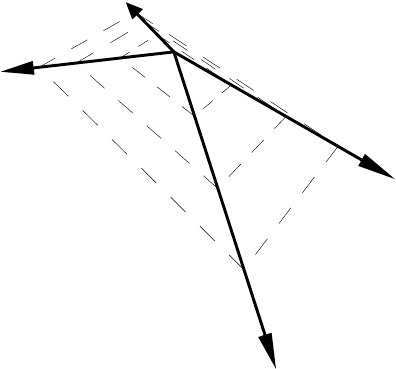}
    }
\hspace{-.5cm}
    \subfigure[\label{badtent}
$X_3=e_r(X_2)$]{
      \includegraphics{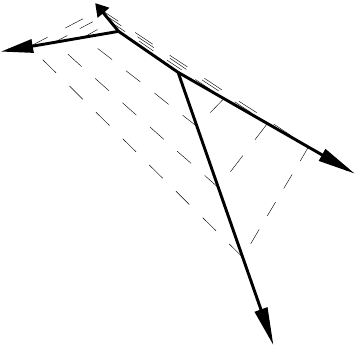}
      }
\caption{$X_1$ is resilient under distance-decreasing transformations.  $X_2$ is not resilient to erosion.  $X_3$ is resilient under distance-increasing transformations.}
 \end{center}
\end{figure}

$X_1$ has an inscribed ball, as depicted in Figure \ref{tent}.  Thus $X_1$ is resilient to erosion by radii less than the radius $r$ of the ball; $X_2=e_r(X_1)$ is depicted in Figure \ref{ibltetra}.  It is not resilient and so has no inscribed ball.

If we erode $X_2$ by the radius $r$, we get the convex body $X_3$ depicted in Figure \ref{badtent}.  It has no inscribed ball, but is nevertheless resilient: the erosion of $X_3$ by any positive radius $r$ satisfies $e_r(X_3)=\sigma(X_3)$ for some similarity transformation $\sigma$ which increases distances.  

Finally: what if in Figures \ref{tent} or \ref{badtent} the dihedral angles between the two pairs of opposite half-spaces had been \emph{equal}?  Like $X_3$, the result $X_4$ would not have an inscribed ball, but would still be resilient.  In fact, for all $r>0$, we would have $\sigma(X_4)=e_r(X_4)$ where $\sigma$ is now a \emph{translation}. 

\subsection{Characterizing \emph{all} resilient convex bodies in $\R^n$}
\label{s.cacc}




We say that a convex body $X$ has an \emph{exscribed ball} of radius $r$ centered at a point $p$ if all of its regular supporting half-spaces are at distance $r$ from $p$ (in particular, $p$ lies outside of all of them). Note that, in general, the distance between $X$ and $p$ may be greater than $r$. 

This concept allows us to extend the characterization in Theorem \ref{t.dcoll} to convex sets for which $e_r(X)=\sigma(X)$, where $r>0$ and $\sigma$ is a distance-\emph{increasing} similarity transformation.

\begin{theorem}
  An (unbounded) convex body $X\sbs \R^n$ satisfies $e_r(X)=\sigma (X)$ for some $r>0$ and a distance-increasing similarity transformation $\sigma$ if and only if $X$ has an exscribed ball.
\label{t.inccoll}
\end{theorem}
\noindent Thus, in particular, the example $X_3$ from Section \ref{R3counter} has an exscribed ball.

The final piece of our characterization of convex resilient sets (covering for example the set $X_4$ from Section \ref{R3counter}) is the following theorem.

\begin{theorem}
    An (unbounded) convex set $X\sbs \R^n$ satisfies $e_r(X)=\sigma (X)$ for some $r>0$ and an isometry $\sigma$ if and only if all regular supporting hyperplanes of $X$ lie at some common dihedral angle $\neq \frac \pi 2$ to some fixed hyperplane.\label{t.isomcoll}
\end{theorem}
\noindent By Lemma \ref{l.convex}, the convexity assumption is superfluous in Theorem \ref{t.isomcoll}.

\subsubsection*{Proof of Theorem \ref{t.inccoll}}

We are striving to show that all of the regular supporting half-spaces of $X$ are at a common distance from some point external to all of them.  We first have to show that there is such an external point.  The following lemma shows that the fixed point of $\sigma$ must be external to every half-space.

\begin{lemma}
  If a convex body $X\sbs \R^n$ satisfies $e_r(X)=\sigma (X)$ for some $r>0$ and a distance-increasing similarity transformation $\sigma$, then the fixed point $p$ of $\sigma$ lies outside of every supporting half-space of $X$. 
\label{l.incout}
\end{lemma}

\begin{figure}
  \begin{center}
    \includegraphics{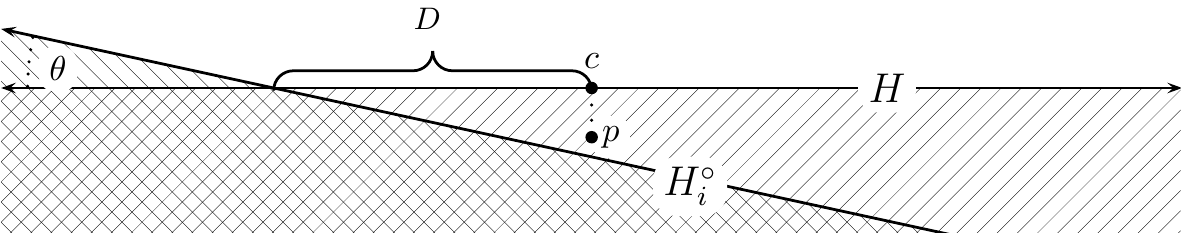}
  \end{center}
  \caption{(Proving Lemma \ref{l.incout}.) \label{f.sepp} }
  
\end{figure}

\begin{proof}
  Write $\sigma=\gamma\circ \phi$, where $\gamma$ is a homothety, $\phi$ is rotation (an isometry of $S^{n-1}$), and $\sigma$, $\gamma$, and $\phi$ all have the same fixed point.  
Let $H$ be a supporting half-space of $X$ which contains the fixed point $p$ of $\sigma$.  Similar to the end of the proof of Theorem \ref{t.dcoll}, our aim is to `separate'  $H$ from $X$ by nearly parallel half-spaces which are also regular supporting half-spaces to $X$.   

As before, let $r_i=\sum_{k=0}^{i-1}\alpha^k r$.  Since $e_{r_i}(X)\sbs H$ for all $i$, $e_{r_i}(X)$ has a (not necessarily regular) supporting half-space $H_i^*$ which is parallel to $H$.  Since $\alpha>1$, the $r_i$'s get arbitrarily large as $i$ increases, so we have for sufficiently large $i$---say, $i>K$ for some $K>0$---that $p$ lies outside of $H_i^*$. Now, since $H_i^*$ is a supporting half-space of $e_{r_i}(X)$, $H_i^\circ=\sigma^{-i}(H_i^*)$ is a supporting half-space to $X$.  The important point is that $p$ must also lie outside of $H_i^\circ$, since it lies outside of $H_i^\star$, and $\sigma$ fixes $p$.

Thus, letting $c$ denote the point on the boundary of $H$ which is closest to $p$, we have that points on the boundary of $H$ within the distance $D$ between $p$ and the intersections of the boundaries of $H$ and $H_i^\circ$ lie outside of $H_i^\circ$, and so outside $X$ (see Figure \ref{f.sepp}).  This distance can be made arbitrarily large by choosing $i$ so that $\theta$ is arbitrarily small, thus we get that there are no points on the boundary of $H$ from $X$, thus $H$ is not supporting to $X$, a contradiction.
%
\end{proof}

We can now prove Theorem \ref{t.inccoll}.  Note that the distances between $p$ and the regular supporting half-spaces of $X$ are bounded by the distance between $p$ and $X$.  Thus, by choosing $i$ sufficiently large, Lemma \ref{t.inccoll} implies that we can ensure that all of the images $\sigma^{-i}(H)$ of regular supporting half-spaces $H$ of $X$ lie within distance $\ep$ of $p$ for any positive $\ep$.  By Observation \ref{o.commutes} we have that $X=e_{r_i/\alpha^i}(\sigma^{-i}(X))$. Thus by Observation \ref{regtreg}, every regular supporting half-space of $X$ is at distance $\frac {r_i}{\alpha^i}$ from a regular supporting half-space of $\sigma^{-i}(X)$; thus, all regular supporting half-spaces of $X$ are at distance $<\frac{r_i}{\alpha_i}+\ep$ from $p$.  Taking the limit as $i$ goes to infinity, we get that any regular supporting half-space of $X$ is at distance $\frac{r}{\alpha-1}$ from the point $p$.\qed



\subsubsection*{Proof of Theorem \ref{t.isomcoll}}

Any isometry can be written as the commutative product $\tau\cdot \phi$ of a translation and rotation (by rotation we mean any isometry with a fixed point).  (A proof of this fact can be found in \cite{regpolys}, p.~217.)  Under the conditions of Theorem \ref{t.isomcoll}, we have that $X$ is collapsible by arbitrarily large radii, so it is easy to see that $\sigma$ can have no fixed point: such a point cannot lie inside $X$, since then sufficiently large radii would erode $X$ to the point where it does not contain it; if it lies outside on the other hand, the distance between it and $X$ increases with erosion.  

Thus $\tau$ is nontrivial.  Since the product $\tau\cdot \phi$ commutes, we have that the rotation $\phi$ must fix the direction $\nu_\tau$ of the translation $\tau$, thus all of the points on a line $L=p+\nu_\tau x,$ $x\in \R$, where $p$ is a fixed point of $\phi$.

Eroding a half-space by any radius gives a half-space similar to the first under a translation.  In fact, similarity can be under a translation in any direction of our choice (except along a vector lying inside any hyperplane parallel to the boundary); the magnitude of the translation then depends just on the radius of collapse and the angle between the direction of the translation and the half-space in question. (We take as the angle between a half-space and a vector the angle between the normal vector of the half-space and the vector).  Thus the pair $r$ and the magnitude $\abs{\tau(0)}$ of the translation $\tau$ together determine a unique angle $\theta_{\tau,r}$, which is the angle a half-space $H$ must lie at relative to the vector $\nu_\tau$ for us to have that $e_r(H)=\tau(H)$.  If we have that $\phi$ is trivial, so that $\sigma=\tau$, then it follows easily (using Observation \ref{regtreg}) that all regular supporting half-spaces of the set $X$ must lie at the angle $\theta_{\tau,r}$ to $\nu_\tau$, as desired: basically, they must all get `pushed' by the erosion in the direction $\nu_\tau$ at the same rate.  We will now show this is true even if $\phi$ is nontrivial.  First we point out that $X$ must intersect $L$: otherwise, the distance between $L$ and $X$ increases with erosion, contradicting that $\sigma$ is an isometry fixing $L$.
%
%
%
This implies that no supporting half-space of $H$ can be at angle $\frac \pi 2$ to $\nu_\tau$, since then, for a sufficiently large $k$, $e_{kr}(H)$ would not intersect the line $L$, and so neither would $e_{kr}(X)$.

If $H$ is a regular supporting half-space to $X$, let $\theta_H$ denote its angle to the vector $\nu_\tau$.  We first claim that we must have that $\theta_H\leq \theta_{\tau,r}$.  Otherwise, if $\theta_H>\theta_{\tau,r}$, we have that $e_r(H)=\tau_H(H)$ where $\tau_H$ is a translation in the direction $\nu_\tau$ but with strictly greater magnitude: $\abs{\tau_H(0)}>\abs{\tau(0)}$. Let $q$ be the point where  $L$ intersects $\delta H$ (we cannot have $L\sbs \delta H$ since, as noted above, their angle cannot be $\frac \pi 2$).   

Observe that the distance between $q$ and $e_{kr}(X)\cap L$ is just the original distance between $q$ and $X\cap L$, plus the magnitude of the translation $\tau^k$:
\begin{equation}
d(q,e_{kr}(X)\cap L)=k\abs{\tau(0)}+d(q,X\cap L).
\label{e.tmoves}
\end{equation}
On the other hand, we have that $e_{kr}(X)\sbs e_{kr}(H)=\tau_H^k(H)$, and for suitably large $k$ we have that $d(q,\tau_H^k(H))>k\abs{\tau(0)}+d(q,X\cap L)$, contradicting line (\ref{e.tmoves}).  (Notice that this part of the argument has not used the regularity of $H$; $X$ cannot have \emph{any} supporting half-spaces at angles $>\theta_{\tau,r}$ from $\nu_\tau$.)

\begin{figure}
  \begin{center}
    \includegraphics{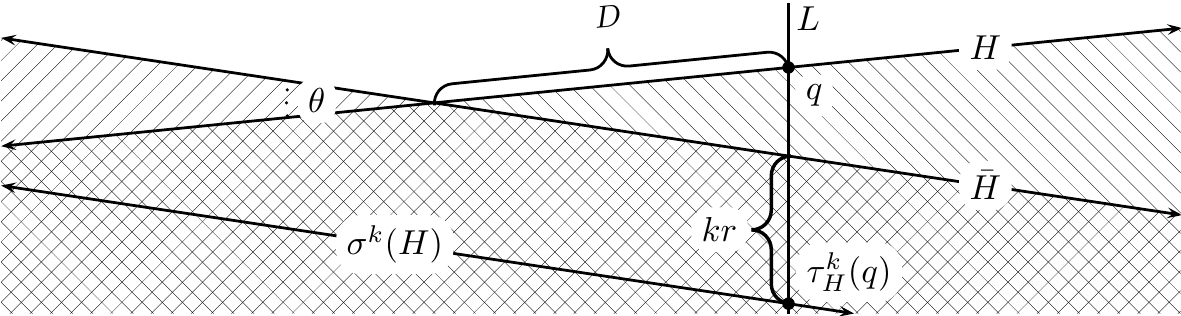}
  \end{center}
  \caption{(Finishing the proof of Theorem \ref{t.isomcoll}.) \label{f.isom} }
  
\end{figure}

Now we wish to show that $\theta_H\geq \theta_{\tau,r}$.  Assume the contrary.  Let again $q$ be the intersection point of $L$ with the boundary of $H$, and again let $\tau_H$ be the translation in the direction $\nu_\tau$ satisfying $e_r(H)=\tau_H(H)$.

Let $H'=\sigma^k(H)$ for some $k$.  We have that $H'$ passes through $\tau_H^k(q)$, which lies at distance $k\abs{\tau_H(0)}$ from $q$.  Observation \ref{regtreg} implies that $X$ has a regular supporting half-space $\bar H$ parallel to $H'$, with $d(\bar H,H')=kr$.  The important point is that the distance between $q$ and $\tau_H^k(q)$ is greater than $kr$, and ever more so as we increase $k$.  Therefore, referring to Figure \ref{f.isom}, we can make the distance $D$ between the boundary of $H$ and the intersection of the boundaries of $H$ and $\bar H$ arbitrarily large by choosing $k$ so that the angle $\theta$ is small.  We conclude that no points on the boundary of $H$ can intersect $X$, a contradiction.\qed

\subsection{Convex sets resilient to expansion}
\label{s.exp}
While in general, asking about erosion and expansion are equivalent by taking complements, it is natural to wonder about the family of \emph{convex} sets resilient to expansion, \emph{i.e.,} convex sets  $X$ for which $E_r(X)=\sigma(X)$ for some $r>0$, and some similarity transformation $\sigma$.  It turns out that this family is not as rich as the corresponding family for erosion.  In particular, any convex expansion-resilient set is also erosion-resilient; moreover, as is not hard to verify, the only bounded sets resilient to expansion are open balls.  The following gives the characterization of all convex sets resilient to expansion:
\begin{theorem}
  A (possibly unbounded) convex open subset of $\R^n$ is resilient to expansion by some radius $r$ if and only if \emph{all} of its supporting half-spaces are at a common distance $R>r$ from some point.
\label{t.expansion}
\end{theorem}
\noindent For the purposes of Theorem \ref{t.expansion}, we say the the supporting half-spaces of an open convex set are the interiors of the supporting half-spaces of its closure.\\
\textbf{Proof of Theorem \ref{t.expansion}:}
  (\textit{Sketch})  Check that the closure of a convex expansion-resilient set is also erosion-resilient (since if $X$ is convex and closed, then we have $e_r(E_r(X))=X$).  Additionally, all of its supporting half-spaces must be regular (since this is true after an expansion).  By examining the characterizations given in Theorems \ref{t.inccoll} and \ref{t.isomcoll}, this rules out that a similarity transformation corresponding to the erosion of $X$ is either distance increasing or distance-preserving (unless $X$ is just a half-space, in which we are done); therefore, the characterization of Theorem \ref{t.dcoll} applies.  Since all of the supporting half-spaces to $X$ must be regular, this completes the proof.
\qed

\section{The nonconvex case}
\label{s.nonconvex}
We begin this section by giving examples of the `strange behaviors' nonconvex (and so unbounded) resilient sets can exhibit; in Section \ref{s.sic}, we show how these sets can be characterized.

\begin{example}
\label{e.path}
Our first example lies in $\R^1$.  It seems `fractal-like', and is resilient to erosion by arbitrarily large radii, even though every component of the set has finite diameter.  Moreover, the set $\{r_i\}$ of radii by which it is resilient to erosion is discrete.  As we will see in the next example, this kind of 1-dimensional resilient set can be used to create sets with similar properties in higher dimensions.

We find it easiest to describe the complement of the set. 
To this end, let 
$
T=\{4,-4,28,-28,196,-196,\ldots \}
$ 
be the set of integers of the form $\pm 4\cdot 7^k,$ $k\geq 0$, and let $A$ be the set of of all numbers $A$ for which 
\[
n=\sum_{t\in T_a} x
\]
 for some subset $T_a\sbs T$.  So, for example, $A$ includes the points 
$$0,\pm 4,\pm 24,\pm 28,\pm 32,\pm 164,\pm 168, \pm 172, \pm 192, \pm 196, \pm 200, \pm 220, \pm 224,\pm 228,$$
 \emph{etc.}  Let now $Y=E_1(A)$ be the expansion of $A$ by radius 1, and let $X=Y^C$ be the complement of $Y$.  Then $X$ is resilient to erosion by a radius $2\cdot 7^k$ for every $k\geq 0$.  The erosion by radius 2 is shown in Figure \ref{f.pathcoll}.

 \begin{figure}
   \begin{center}
     \includegraphics{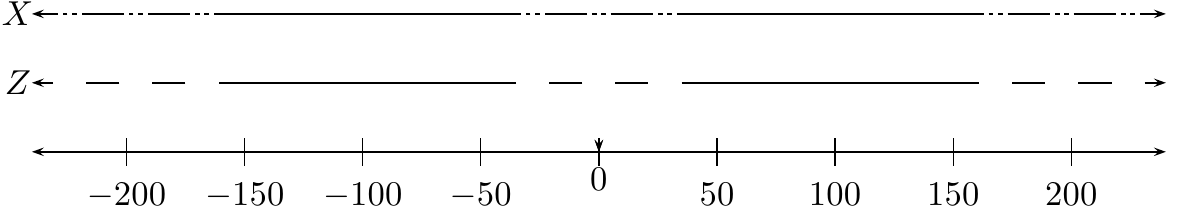}
   \end{center}
\caption{Shown are portions of $X$ and its erosion $Z=e_2(X)$ by the radius 2.\label{f.pathcoll}}
 \end{figure}

We will sketch the proof of this fact in terms of the expansion of $Y$, rather than the erosion of $X$, so we will prove that $E_2(Y)\sim Y$; in fact, that $E_2(Y)=7\cdot Y$, where $[\cdot]$ here indicates pointwise multiplication.  
This is equivalent to showing that $E_3(A)=7\cdot E_1(A)$.

Let $p\in E_1(A)$.  Then $d(p,a)<1$ for some 
\begin{equation}
a=\sum_{t\in T_a} t,
\label{aTa}
\end{equation}
where $T_a\sbs T$.  Multiplying by $7$, we get
that all points $p'$ of $7\cdot E_1(A)$ are within $<7$ units of a point $a'$ of the form 
\begin{equation}
  a'=\sum_{t\in T'} t
\label{a'T'}
\end{equation}
for some $T'\sbs 7\cdot T$.  But $T=(7\cdot T)\dot\cup \{4,-4\}$.
Thus we have that $p'$ must be within distance $<7-4=3$ of a point of $A$; thus we have that  $7\cdot E_1(A)\subseteq E_3(A).$

For the other direction, let now $p\in E_3(A)$.  So $d(p,a)<3$ for some $a$ satisfying line (\ref{aTa}) for some $T_a\sbs T$.  Thus $d(p,a')<7$ for some $a'$ satisfying line (\ref{a'T'}) for $T'\sbs T\stm \{4,-4\}=7\cdot T$.  Thus $\frac 1 7 p$ is at distance $<1$ from a point of $T$, and so we have shown that $E_3(A)\subseteq 7\cdot E_1(A)$.  Combining the two directions, we have $E_3(A)=7\cdot E_1(A)$, as desired.

\begin{figure}[p]
  \begin{center}
    \includegraphics{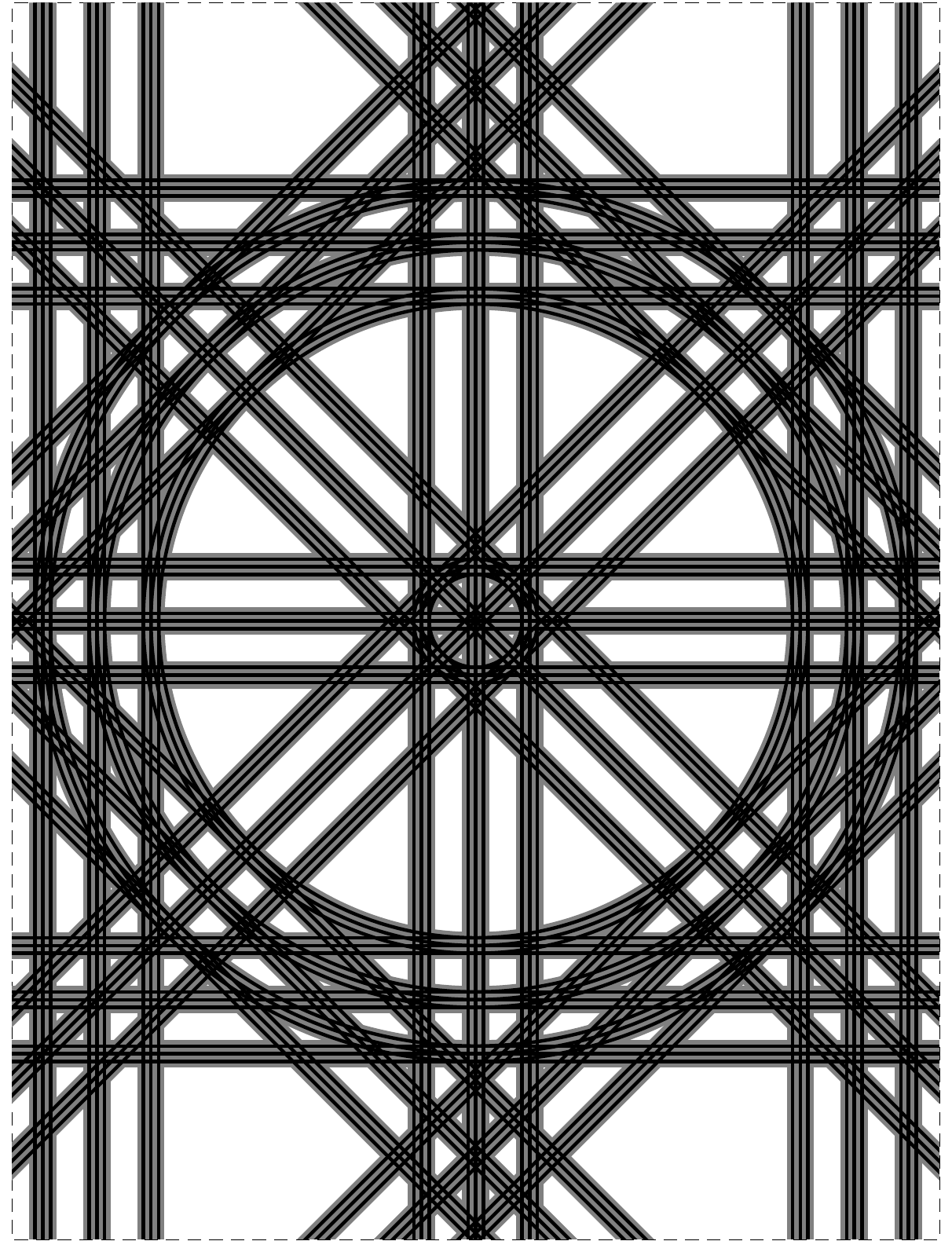}
  \end{center}
  \caption{An example of a fractal-like unbounded expansion-resilient set from $\R^2$---it is cropped at the dashed box.  Since it is expansion-resilient, its complement is erosion-resilient.  The gray area is the region added to the set upon an expansion by the smallest radius by which it is expandable.  
The corresponding similarity transformation is a distance-increasing homothety fixing the center of the region shown.\label{plaid}}
\end{figure}
\end{example}

\begin{example}
\label{e.int}
Recall that the erosion of an intersection is the intersection of the erosions.  Thus, taking the intersection of sets $X_i$ all satisfying $\sigma(X_i)=e_r(X_i)$ for fixed $r,\sigma$, gives another resilient set.  Equivalently, we can make examples of expandable sets from other suitable expandable sets by taking unions instead of intersections.  Using this basic method, we can create in $\R^n$ diverse classes of examples of unbounded sets displaying the same properties as the example we have given in $\R^1$.   See for example Figure \ref{plaid}, which illustrates an unbounded expandable set in $\R^2$ constructed from four rotated copies of the product $Y\times \R$, together with the set $\{z\in \R^2\textrm{ s.t. } \lVert z \rVert \in Y\}$, a circular variant of $Y$.
\end{example}

\begin{example}
\label{e.spiral}
This example is based on the classical logarithmic spiral, given in polar coordinates by $R(\theta)=ae^{b\theta}$.  The logarithmic spiral has the remarkable property that it is self-similar  under similarity transformations (centered at the origin) which take on all scaling factors $\alpha>0$; the rotation in the similarity transformation varies continuously with the choice of $\alpha$.

To get our example, we modify this spiral by giving it a `truncated' logarithmic thickness: more precisely, let $R_0(\theta)=ae^{b\theta}$ and $T_0(\theta)=e^{b\theta}-1$, and let 
\begin{equation}
  S_1=\bigcup_{\theta}\bar B(T_0(\theta),(R_0(\theta);\theta)),
\label{defS0}
\end{equation}
where $\bar B(T_0(\theta),(R_0(\theta);\theta))$ is the closed ball of radius $T_0(\theta)$ about the point $(R_0(\theta);\theta)$, given in polar coordinates.  Then $S_1$ is resilient to erosion by \emph{any} radius.  In general, the similarity transformation includes a rotation which varies continuously with the choice of the radius of erosion.  ($S_1$ is shown in Figure \ref{f.spiral}, with two of its erosions).  The proof of the resiliency of $S_1$ will follow from our characterization of resilient sets in the next section.
\end{example}

%
\begin{figure}
\begin{center}
  \includegraphics{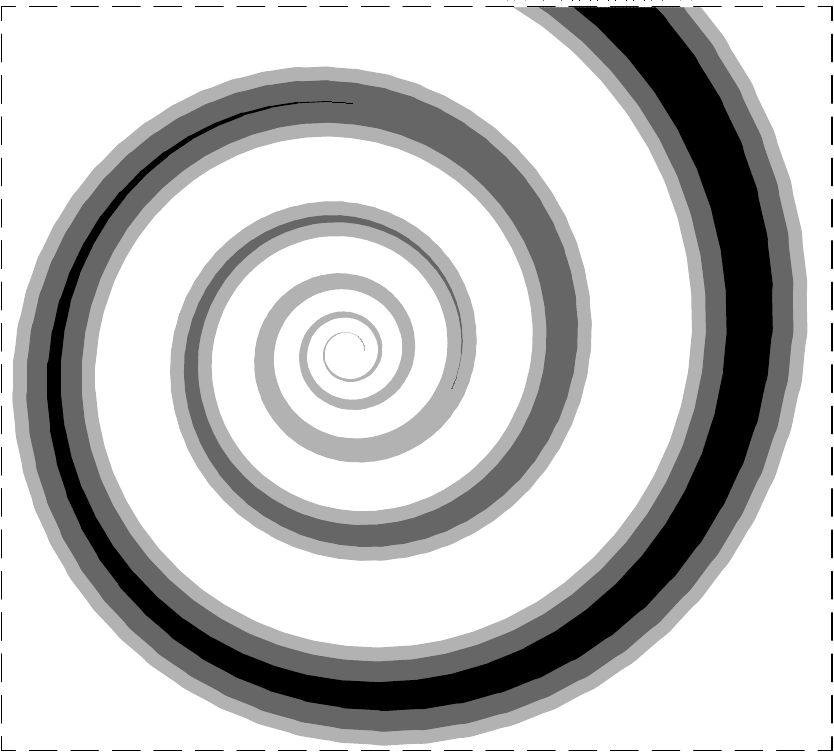}
\end{center}
  \caption{$S_r$ is a logarithmic spiral with `truncated' logarithmic thickness, and is resilient to erosion by any radius; the corresponding similarity transformation includes a rotation by an angle determined by the radius of erosion.\label{f.spiral}  (Two arbitrary erosions are shown.)}
  
\end{figure}



\begin{example}
\label{x.discspiral}
Fix some similarity transformation $\sigma$ centered at the origin which increases distances and includes an irrational rotation. Fix the body $T$ of some rectangle, and let $Q_0$ be the union of all the images $\sigma^i(T)$, over all integers $i$.  Then the erosion $Q_r=e_r(Q_0)$, shown in Figure \ref{f.discspiral}, is resilient to erosion by a discrete set of radii, starting with the radius $r(\alpha-1)$; each erosion induces a similarity transformation which includes an irrational rotation.  Again, the proof will follow from Theorem \ref{t.sic}.
\end{example}


\begin{figure}
  
  \begin{center}
    \includegraphics{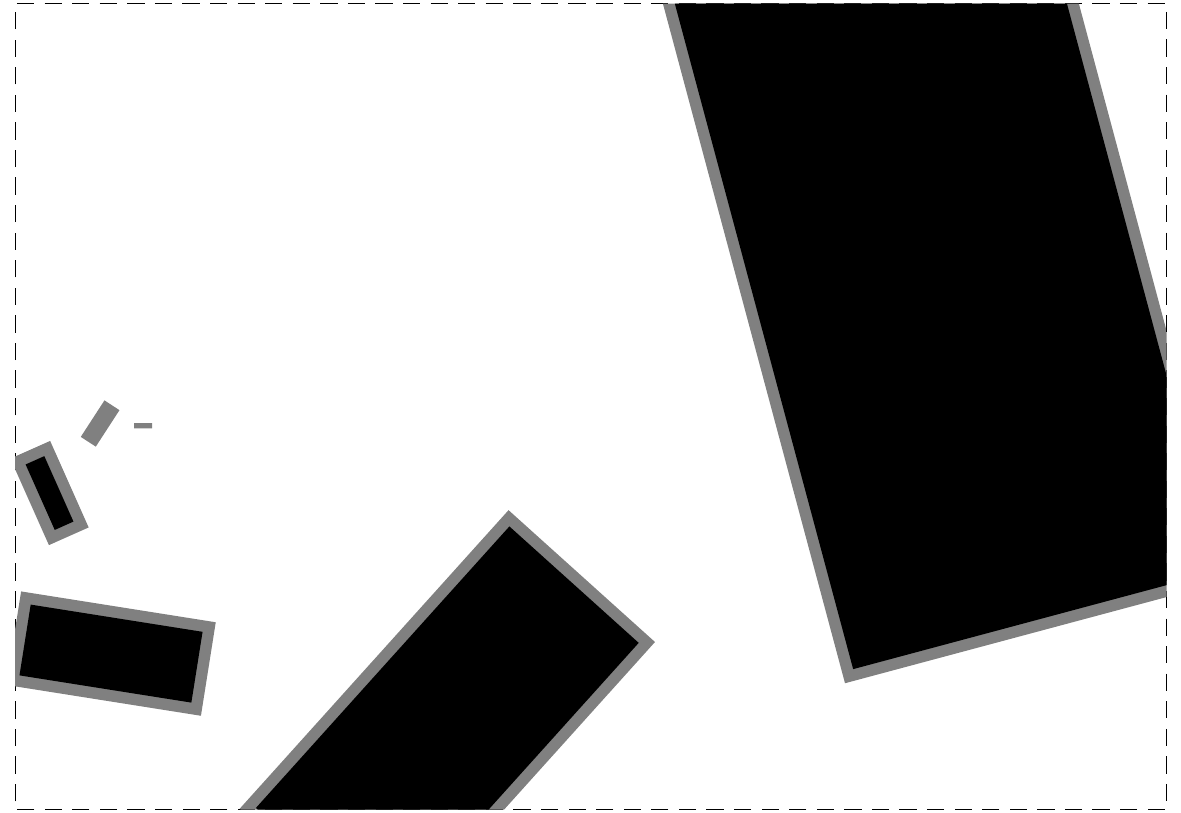}
  \end{center}

  \caption{\label{f.discspiral}$Q_r$ is a discrete `spiraling' resilient set.  It is shown with its \emph{second} erosion in black.}
\end{figure}


\subsection{Characterizing nonconvex resilient sets}
\label{s.sic}

To finish our characterization of resilient sets, we need to cover nonconvex sets; by Lemma \ref{l.convex}, it is enough to cover the case where the similarity transformation corresponding to erosion is distance-increasing.  In light of the examples in the previous section, it is clear that the characterization for this case will not be of the simple geometric type given for the distance decreasing and isometric cases.  Instead, the characterization will show that any resilient set whose corresponding similarity transformation is distance increasing can be constructed essentially in the same way as we constructed Example \ref{x.discspiral}, and moreover, all such constructions lead to resilient sets.  Let's make this more precise:
\begin{definition}
  A set $W\sbs \R^n$ is \emph{scale-invariant} if we have $W=\sigma(W)$ for a similarity transformation $\sigma$ with scaling factor $\alpha>1$.
\label{d.si}
\end{definition}
Notice now that $Q_0$ from Example \ref{x.discspiral} is a scale-invariant set, thus $Q_r$ was constructed as the erosion of a scale-invariant set.  This feature characterizes all sets which are resilient under distance-increasing similarity transformations:
\begin{theorem}
  A set $X$ satisfies $e_r(X)=\sigma(X)$ for some distance-increasing similarity transformation $\sigma$ if and only if we have $X=e_{r/(\alpha-1)}(W)$ for some set $W$ which is scale-invariant under $\sigma$.
\label{t.sic}
\end{theorem}
In particular, note that the spiral from Figure \ref{f.spiral} is the erosion of a scale-invariant logarithmic spiral $S_0$ with thickness $T_0(\theta)=e^{b\theta}$.  Theorem \ref{t.sic} also implies that Examples \ref{e.path} and \ref{e.int} are the erosions of scale-invariant sets.  In these cases, the original self-similar sets are dense in $\R^1$ and $\R^2$, respectively.   Notice that Theorem \ref{t.sic} overlaps with Theorem \ref{t.inccoll} in its coverage of convex sets resilient to erosion with corresponding distance-increasing similarity transformations.  For example, in the case of the resilient set $X_3$ (Figure \ref{badtent}) discussed in Section \ref{R3counter}, the corresponding scale-invariant set is the set $X_2$ (Figure \ref{ibltetra}).

For the proof of Theorem \ref{t.sic}, let's first see why the `if' direction is true.  We have that $W=\sigma(W)$, where $\sigma$ has scaling factor $\alpha>1$.  Applying Observation \ref{o.commutes} with some radius  $r'$ to the set $W$ gives us that $\sigma(e_{r'}(W))=e_{\alpha r'}(\sigma(W))$.  Thus $e_{\alpha r'}(W)=\sigma(e_{r'}(W))$.  But we have (by the triangle inequality) that $e_{\alpha r'}(W)=e_{r'(\alpha-1)}(e_{r'}(W))$; thus we have that
$e_{r'(\alpha-1)}(e_{r'}(W))=\sigma(e_{r'}(W)).$
Thus, letting $r=r'(\alpha-1)$, we have that the set $X=e_{r'}(W)$ is resilient to erosion by the radius $r$.

For the other direction, we need to first construct the set $W$ which should be scale invariant and give rise to $X$ under erosion.  Observe that since $X$ is similar to $e_r(X)$,  Observation \ref{o.commutes} gives us that $X=e_{r/\alpha}(\sigma^{-1}(X))$.  In fact, we have the following observation, similar to Observation \ref{sup}:
\begin{observ}
  If $e_r(X)=\sigma(X)$ for some similarity transformation $\sigma$ with scaling factor $\alpha$,  then for all $i\geq 1$ and 
\begin{equation}
r_{-i}=\sum\limits_{1\leq k\leq i} \frac r {\alpha^k},
\label{e.rmis}
\end{equation}
we have $X=e_{r_{-i}}(\sigma^{-i}(X))$.  \qed
\label{supfrom}
\end{observ}
\noindent Motivated by this observation, we set $W$ as
\begin{equation}
W=\bigcup_{k\geq 1} \sigma^{-k}(X);
\label{e.W}
\end{equation}
note that $W$ is thus a scale-invariant set: we have $\sigma(W)=W$.

To complete the proof, it just remains to check that $e_{r/(\alpha-1)}(W)=X$.  Note that $\frac r {\alpha-1}=\sup_{(i\geq 1)} r_{-i}$, thus we certainly have that $e_{r/(\alpha-1)}(W)\supset X$; otherwise, some point of $X$ must be at distance $r_0<\frac r {\alpha-1}$ from the complement of $W$.  We would then have for some sufficiently large $i$ that $r_0<r_{-i}$, a contradiction since $e_{r_{-i}}(\sigma^{-i}(X))=X$, yet $\sigma^{-i}(X)\sbs W$.

It remains to show that $e_{r/(\alpha-1)}(W)\sbs X$.  It suffices to show that any point in the complement $X^C$ is at distance $< \frac r {\alpha -1}$ from the complement $W^C$.

Since $X=e_{r_{-1}}(\sigma^{-1}(X))$, we have by the definition of the erosion operation that any point $y_0\in X^C$ must lie at distance $<\frac r \alpha$ from some point $y_1$ in the complement of $\sigma^{-1}(X)$.  Similarly, $y_1$ must lie at distance $<\frac r {\alpha^2}$ from a point $y_2$ in the complement of $\sigma^{-2}(X)$.  Continuing in this manner, we get a sequence $\{y_i\}_{i\geq 0}$ where, for each $i$, $y_i$ lies outside of $\sigma^{-i}(X)$, and the distance $d(y_{i-1},y_i)$ is less than $\frac r {\alpha^i}$.  Since $\alpha>1$, this is a Cauchy sequence, and there is some limit point $y_\infty.$  We have that
\[
d(y_0,y_\infty)<\sum_{k\geq 1}\frac{r} {\alpha^k}=\frac r {\alpha-1},
\]
thus it just remains to check that $y_{\infty}$ does not lie in $W$.  If it does, then by line (\ref{e.W}) we have that $y_\infty\in \sigma^{-i}(X)$ for some $i$.  But then, for some sufficiently large $j$, we have that $r_{-j}-r_{-i}>d(y_j,y_\infty)$, a contradiction since $e_{r_{-j}-r_{-i}}(\sigma^{-j}(X))=\sigma^-i(X)$, and $y_j\notin \sigma^{-j}(X)$.\qed

\subsection{Fractals and erosion}
\label{s.fractals}
Theorem \ref{t.sic} shows a direct correspondence between scale-invariant sets and  resilient sets with corresponding distance-increasing similarity transformations.  Although there is a connection between fractals and scale-invariant sets, these concepts are certainly not the same.  The family of scale-invariant sets includes sets such as lines, half-spaces, and the figure $X_2$ from Figure \ref{ibltetra} which are certainly not `fractals'.  Moreover, the definition of a fractal given in \cite{mandel} by Mandelbrot---a set whose Hausdorff dimension is strictly less than than its covering dimension---includes many sets which are not actually `self-similar' in any exact sense.

Nevertheless, an important class of fractals (including well-known examples such as Koch's snowflake and Sierpi\'nski's triangle) is that of fractals generated by `Iterated Function Systems'.  Our goal in this section is to briefly point out how members of this class can give rise to scale-invariant sets.  For more background on Iterated Function Systems, see, for example, \cite{fg}.

An \emph{Iterated Function System} is a list $f_1,f_2,\dots f_t$ of distance-decreasing similarity transformations of $\R^n$.  Its \emph{ratio list} is the list $\alpha_1, \alpha_2,\dots,\alpha_t$ of scaling factors of the similarity transformations.  Any iterated function system has a unique nonempty compact \emph{invariant} set $K$ satisfying
\begin{equation}
  K=f_1(K)\cup f_2(K)\cup \cdots \cup f_t(K).
\end{equation}
The \emph{similarity dimension} of the iterated function system is the solution $s$ to the equation
\begin{equation}
  \sum_{i=1}^t a_i^s=1.
\end{equation}
We will see shortly that under a certain condition, the similarity dimension is the same as the Hausdorff dimension, so that $K$ will be a fractal in the sense of Mandelbrot's definition so long as $s$ is strictly less than the covering dimension (in particular, if it is not an integer).  Note that Sierpi\'nski's triangle is the invariant set of an iterated function system with ratio list $(\frac 1 2, \frac 1 2, \frac 1 2)$ and so has similarity dimension $\frac {\log 3}{\log 2}$.  Koch's curve is the invariant set of an iterated function system with ratio list $(\frac 1 3, \frac 1 3, \frac 1 3, \frac 1 3)$, thus has similarity dimension $\frac {\log 4}{\log 3}$.

We are trying to show how to get scale-invariant sets from fractals generated by an iterated function system.  Of course, given \emph{any} set $X$ and a similarity transformation $\sigma$, the set 
\[
SI_\sigma(X)=\bigcup_{k\in \Z}\sigma^k(X)
\]
is a scale-invariant set.  In general, however, $SI_\sigma(X)$ may not seem very related to the original set $X$; it may very well turn out to be the whole space, for example.

In the case of the invariant set of an iterated function system where we let $\sigma=f_1$, say, we have that 
\begin{equation}
  SI_{f_1}(K)=\bigcup_{k\in \Z}f_1^{k}(K)=\bigcup_{k\in \Z^-}f_1^{k}(K)=K\cup \bigcup_{\substack{k\geq 0\\ 2\leq i\leq k}}f_1^{-k}\circ f_i(K),
\end{equation}
suggesting that $SI_{f_1}(K)$ will retain some of the appearance of the set $K$.  The problem is that, in general, $K$ and
\[
N(K)=\bigcup_{\substack{k\geq 0\\ 2\leq i\leq k}}f_1^{-k}\circ f_i(K)
\]
may overlap, thus we are not guaranteed that $SI_{f_1}(K)$ shares the `structure' of $K$.  The \emph{open set condition} (OSC) on iterated function systems, introduced by Moran \cite{osc}, is a condition which controls the extent of this overlap: it requires that there is an open set $U$ with $f_i(U)\sbs U$ for all $1\leq i\leq k$, and $f_i(U)\cap f_j(U)=\varnothing$ for $i\neq j$.  One important consequence of the open set condition is that the similarity dimension of the iterated function system must coincide with the Hausdorff dimension of the corresponding invariant set; in particular, this allows the dimension to be easily computed in many cases (OSC is satisfied for the well-known examples of mathematical fractals, such as Sierpi\'nski's Triangle and the Koch curve).  A recent result of Bandt, et al. (Corollary 2 in \cite{conosc}) implies that if an iterated function system satisfies the OSC, then there is an open set $V$ which intersects the invariant set $K$ nontrivially and is disjoint from $N(K)$; thus the intersection $V\cap K$ will be identical to the intersection $V\cap SI_{f_1}(K)$.  Thus OSC ensures that $K$ and $SI_{f_1}(K)$ share the same small-scale structure, in a certain sense.  

The preceding remarks imply that any fractal coming from an iterated function system satisfying the open set condition gives rise to a scale-invariant set which retains the small-scale structure of the original fractal.  Figure \ref{f.iST} shows a portion of the scale-invariant extension of Sierpi\'nski's Triangle, and Figure \ref{f.koch} shows a portion of the scale-invariant extension of the Koch curve.  Since as fractals these sets have empty interior, it is necessary to take complements before taking the erosion to get a nonempty resilient set.  Figure \ref{f.Sr} shows the resulting resilient set produced from Sierpi\'nski's Triangle.  In the case of the unbounded version of Koch's curve, the complement consists of two components, which are each scale-invariant and thus can separately give rise to scale-invariant sets.  Figure \ref{f.ckoch} shows the resilient set produced by taking the erosion of one of these components.

\psset{unit=.75cm}
\begin{figure}
  \subfigure[\label{f.iST} Sierpi\'nski's Triangle can be extended ad-infinitum to create a scale-invariant set.]{
\includegraphics{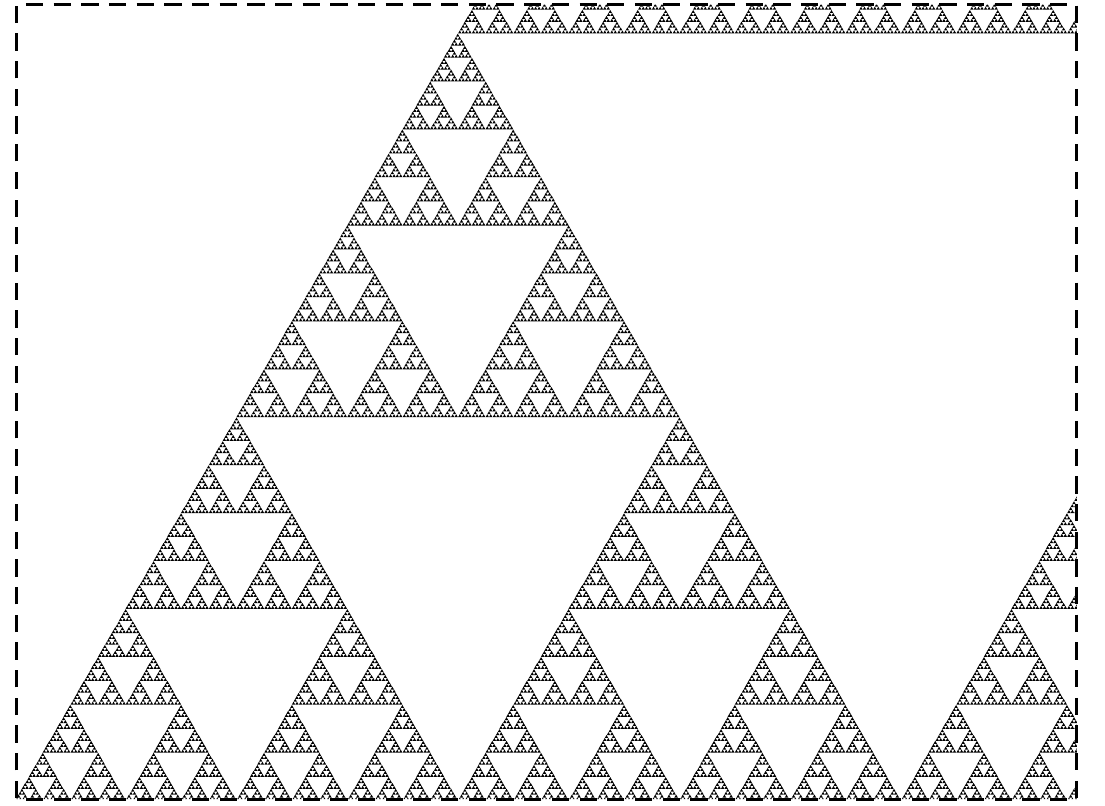}
}

  \subfigure[\label{f.Sr}$S_r=e_r(S)$ is the erosion of the complement of the unbounded version of Sierpi\'nski's Triangle, and is resilient to erosion by the radius $r$.  The area removed by its first erosion is shown in gray.]{
\includegraphics{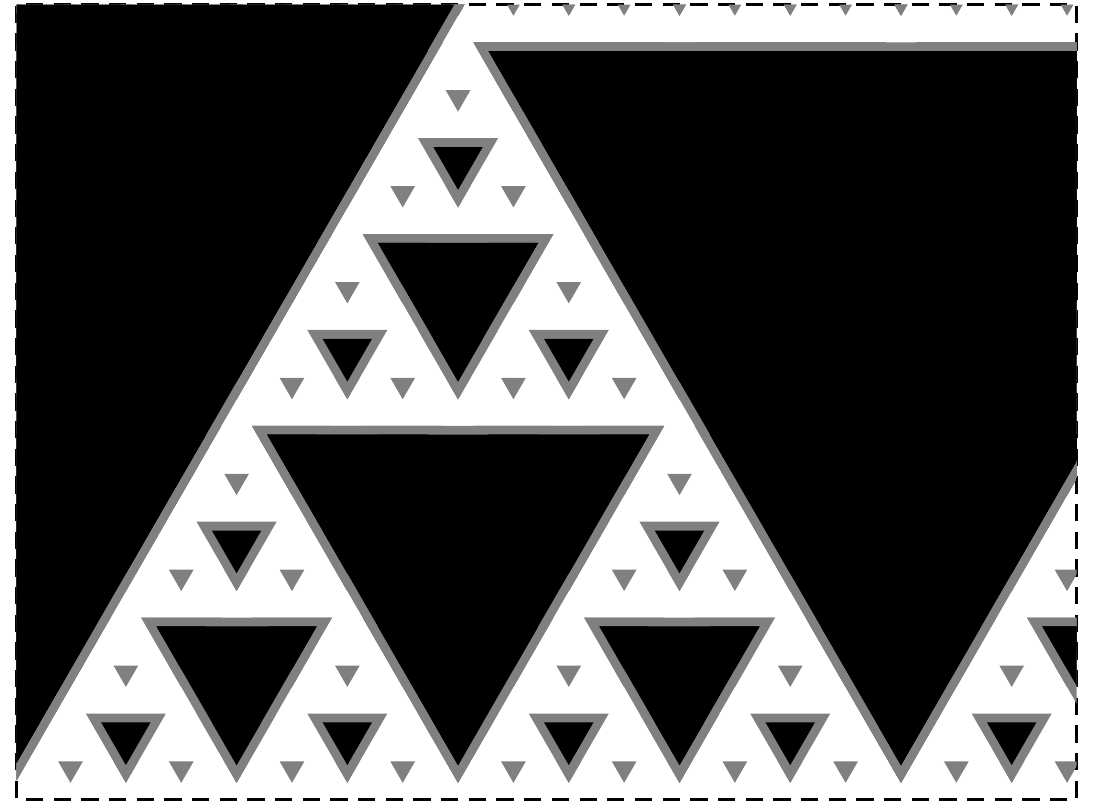}
}

\caption{Getting a resilient set from a fractal.}
\end{figure}

\section{Further Questions}
\label{s.Q}

Since the concepts of erosion, expansion, and similarity make sense in an arbitrary metric space, the question of resiliency could be studied in a wide range of settings. (Note that our proof of Theorem \ref{t.sic} is actually valid in any complete metric space.) 

There is a natural line of inquiry in Euclidean space suggested by Theorem \ref{t.sic}, however.  Since this theorem makes its characterization in terms of scale-invariant sets, it is natural to wonder about the `behavior' of such sets.  In this case, a natural line of attack seems to be from the standpoint of their transformation groups.  Define the \emph{self-similarity group} of a set $X$ as the group of similarity transformations $\sigma$ satisfying $X=\sigma(X)$.  What can we say about which groups appear in this way?

The case of isometric transformation groups  has received considerable attention because of its applications in crystallography (see \emph{e.g.,} \cite{3dsg}, \cite{4dsg}).  Because of the application, attention is restricted in that case to discrete groups.  For our application to resilient sets, there are relevant scale-invariant sets like the thickened logarithmic spiral $S_0$ whose self-similarity groups really are not discrete.  Nevertheless, the following question puts a reasonable restriction on the self-similarity groups, which is necessary anyway for the set to give rise to nonempty resilient sets:

\begin{question}
  Which groups  of similarity transformations occur as the self-similarity groups of subsets of $\R^n$ with nonempty interior?
  \label{igroupq}
\end{question}
\noindent It seems likely that Question \ref{igroupq} is interesting even for small values of $n$.

\end{document}